\documentclass[12pt,a4paper]{article}
\usepackage{style,times}
\usepackage{hyperref}
\usepackage{url}
\usepackage{graphicx}
\usepackage[T1]{fontenc}
\usepackage[bitstream-charter]{mathdesign}

\usepackage[latin1]{inputenc}							
\usepackage{amsmath}	
\usepackage{latexsym,xypic}								
\usepackage{geometry}
\usepackage{longtable}
\usepackage{array}
\usepackage{fancyvrb,shortvrb}
\usepackage[ruled,vlined]{algorithm2e}

\newtheorem{theorem}{Theorem}
\newtheorem{definition}[theorem]{Definition}

\newtheorem{corollary}[theorem]{Corollary}

\newtheorem{remark}[theorem]{Remark}

\newtheorem{proposition}[theorem]{Proposition}

\newenvironment{proof}[1][Proof]{\noindent\textbf{#1.} }{\ \rule{0.5em}{0.5em}}
\newtheorem{examples}[theorem]{Examples}

\title{Isoclinic Lie Crossed Modules}

\author{	E. Ilgaz, A. Odaba\c{s} and E. \"{O}. Uslu	\vspace*{10pt}}

\begin{document}

\maketitle

\begin{abstract}
We introduce the notion of isoclinism among crossed modules of Lie algebras, which will be called "Lie crossed modules" hereafter, and investigate some basic properties. Additionally, we introduce the notion of class preserving actor of a Lie crossed module and the relation with isoclinism.
\end{abstract}



\section{Introduction}

The notion of isoclinism, which is an equivalence relation weaker than
isomorphism, was introduced in \cite{PH}, and detailed in \cite{MH}. This
notion has since been further studied by a number of authors, \cite%
{MRJ,RM,HM,FP,ARS,JT}. A Lie algebra analogue of isoclinism was introduced
in \cite{KM}. The structure of Lie algebras occurring in an isoclinism
family has been studied in \cite{ARS} and also some applications given in 
\cite{BS, SAM, ARS}. Furthermore, in \cite{SM} the notion generalized to the
notion of $n$-isoclinism, that is the isoclinism with respect to the variety
of nilpotent Lie algebras of class at most $n$.

A Lie crossed module $L$ is a Lie algebra homomorphism $d:L_{1}%
\longrightarrow L_{0}$ with an action of $L_{0}$ on $L_{1}$ satisfying
certain conditions. We refer \cite{C1, C2, C3}, for a comprehensive and
detailed research about the notion. The notion is thought as the 2-dimensional Lie algebras
in categorical viewpoint, was introduced in \cite{KL}, in which they proved
that the third dimensional cohomology of Lie algebras classifies Lie crossed
modules. In this work, we introduce the notion of isoclinism among Lie
crossed modules which gives rise to a new classification weaker than
isomorphism , and the resulting equivalence classes called isoclinism
families. Also we introduce the class preserving actor of a Lie crossed module by inspiring from the notion of class preserving automorphisms of a group.   

In order to get our goals, we organize the paper as follows; In Section 2,
we recall some needed results about Lie crossed modules and introduce class preserving actor of a Lie crossed module which will be used in the sequel of the paper. In Section 3, we introduce the notion of isoclinism for Lie crossed modules and establish the
basic theory. As expected, we give the compatibility of this definition with
nilpotency, solvability and class preserving actors of Lie crossed modules.%
\newline

\section{Preliminaries}

In this section we recall some needed material about Lie crossed module. See 
\cite{C2,C3}, for a comprehensive research about the notion. Additionally,
we introduce the notion of \textquotedblleft class preserving actor of a Lie
crossed module\textquotedblright .

We fix a field $\mathbb{K}$ and assume all Lie algebras
to be over $\mathbb{K}\mathbf{.}$

\subsection{Isoclinism among Lie algebras}

\begin{definition}
	\cite{KM} Let $\mathfrak{g}$ and $\mathfrak{h}$ be Lie algebras. $\mathfrak{g%
	}$ and $\mathfrak{h}$ are said to be \textit{isoclinic} if there exist
	isomorphisms $\eta :\mathfrak{g}/Z(\mathfrak{g})\longrightarrow \mathfrak{h}%
	/Z(\mathfrak{h})$ and $\xi :[\mathfrak{g},\mathfrak{g}]\longrightarrow
	\lbrack \mathfrak{h},\mathfrak{h}]$ between central quotients and derived
	subalgebras, respectively, such that, the following diagram 
	\[
\xymatrix { {{\mathfrak{g}}/{Z(\mathfrak{g})}} \times {{\mathfrak{g}}/{Z(\mathfrak{g})}} \ar[rr]^-{c_{\mathfrak{g}}} \ar[d]_{{\eta}\times{\eta}}  & & [\mathfrak{g},\mathfrak{g}] \ar[d]^{\xi} \\ {\mathfrak{h}}/{Z(\mathfrak{h})}\times{\mathfrak{h}}/{Z(\mathfrak{h})} \ar[rr]_-{c_{\mathfrak{h}}}  & & [\mathfrak{h},\mathfrak{h}] }
	\]%
	is commutative where $c_{\mathfrak{g}},c_{\mathfrak{h}}$ are commutator maps
	of Lie algebras. The pair $(\eta ,\xi )$ is called an isoclinism from $%
	\mathfrak{g}$ to $\mathfrak{h},$ and denoted by $(\eta ,\xi ):\mathfrak{g}%
	\sim \mathfrak{h}.$
\end{definition}

\begin{examples}
	\ \newline
	(1) All abelian Lie algebras are isoclinic to each other. The pairs $(\eta
	,\xi )$ consist of trivial homomorphisms.\newline
	(2) Every Lie algebra is isoclinic to a stem Lie algebra ( a Lie algebra
	whose center is contained in its derived subalgebra).\newline
\end{examples}

\subsection{Lie crossed modules}

\begin{definition}
	\cite{KL} \textit{A Lie crossed module} is a Lie algebra homomorphism 
	\[
	d:L_{1}\longrightarrow L_{0} 
	\]%
	with a Lie action of $L_{0}$ on $L_{1}$ written $(l_{0},l_{1})\mapsto $ $%
	[l_{0},l_{1}],$ for $l_{0}\in L_{0},l_{1}\in L_{1}$ satisfying the following
	conditions: 
	\[
	\begin{array}{llll}
	1) & d([l_{0},l_{1}]) & = & [l_{0},d(l_{1})], \\ 
	2) & [d(l_{1}),l_{1}^{\prime }] & = & [l_{1},l_{1}^{\prime }],%
	\end{array}%
	\]%
	for all $l_{0}\in L_{0},$ $l_{1},l_{1}^{\prime }\in L_{1}.$
\end{definition}

In general such a crossed module is denoted by $L:L_{1}\overset{d}{
	\longrightarrow }L_{0}.$

We recall the following examples from \cite{C3}.

\begin{examples}
	\ \newline
	(1) $\mathfrak{g}\overset{ad}{\longrightarrow }Der(\mathfrak{g})$ is a
	crossed module for any Lie algebra $\mathfrak{g.}$\newline
	(2) If $\mathfrak{h}$ is an ideal of $\mathfrak{g,}$ then $\mathfrak{g}$
	acts on $\mathfrak{h}$ via adjoint representation and $\mathfrak{h}\overset{%
		inc.}{\hookrightarrow }\mathfrak{g}$ is a crossed module. In particular, we have the crossed
	modules $0\overset{inc}{\hookrightarrow }\mathfrak{g}${\ and }$\mathfrak{g}${%
		$\overset{id}{\longrightarrow }\mathfrak{g}.$\newline
		(3) If }$\mathfrak{h}$ is a $\mathfrak{g}$-module, then $\mathfrak{h}${$%
		\overset{0}{\longrightarrow }\mathfrak{g}$ is a crossed module.}
\end{examples}

\begin{definition}
	\cite{C1} A Lie crossed module $L:L_{1}\overset{d}{\longrightarrow }L_{0}$
	is called \textit{aspherical} if $\ker \left( d\right) =0,$ and is called 
	\textit{simply connected} if $coker\left( d\right) =0$.
\end{definition}

\textit{A morphism} between Lie crossed modules $L:L_{1}\overset{d}{%
	\longrightarrow }L_{0}$ and $L^{\prime }:L_{1}^{\prime }\overset{d^{\prime }}%
{\longrightarrow }L_{0}^{\prime }$ is a pair $(\alpha ,\beta )$ of Lie
algebra homomorphisms $\alpha :L_{1}\longrightarrow L_{1}^{\prime },$ $\beta
:L_{0}\longrightarrow L_{0}^{\prime }$ such that $\beta d=d^{\prime }\alpha $
and $\alpha ([l_{0},l_{1}])=[\beta (l_{0}),\alpha (l_{1})],$ for all $%
l_{0}\in L_{0},$ $l_{1}\in L_{1}$. Consequently, we have the category $%
\mathbf{XLie}$ whose objects are Lie crossed modules.

A Lie crossed module $M:M_{1}\overset{d^{\prime }}{\longrightarrow }M_{0}$
is a \textit{subcrossed module} of a crossed module $L:L_{1}\overset{d}{%
	\longrightarrow }L_{0}$ if $M_{1},$ $M_{0}$ are Lie subalgebras of $L_{1}$, $%
L_{0},$ respectively, $d^{\prime }=d|_{M_{1}}$ and the action of $M_{0}$ on $%
M_{1}$ is induced from the action of $L_{0}$ on $L_{1}.$ Additionally, if $%
M_{1}$ and $M_{0}$ are ideals of $L_{1}$ and $L_{0},$ respectively, $%
[l_{0},m_{1}]\in M_{1}$ and $[m_{0},l_{1}]\in M_{1},$ for all $l_{0}\in
L_{0},$ $l_{1}\in L_{1},$ $m_{0}\in M_{0},$ $m_{1}\in M_{1}$ then $M$ is
called an \textit{\ ideal} of $L.$ Consequently, we have the \textit{%
	quotient crossed module} $L/M:L_{1}/M_{1}\overset{\overline{d}}{%
	\longrightarrow }L_{0}/M_{0}$ with the induced boundary map and action.

Let $M:M_{1}\overset{d}{\longrightarrow }M_{0}$ be a subcrossed module and $%
N:N_{1}\overset{d}{\longrightarrow }N_{0}$ be an ideal of $L:L_{1}\overset{d}%
{\longrightarrow }L_{0}.$ Then, we have an isomorphism such that

\[
\frac{M}{M\cap N}\cong \frac{M\oplus N}{N}, 
\]%
where $M\cap N:M_{1}\cap N_{1}\overset{d}{\longrightarrow }M_{0}\cap N_{0}$
and $M\oplus N:M_{1}\oplus N_{1}\overset{d}{\longrightarrow }M_{0}\oplus
N_{0}.$

See \cite{C1} for details.

Now we recall the construction of an actor of a Lie crossed module from \cite{C3}

Let $L:L_{1}\overset{d}{\longrightarrow }L_{0}$ be a Lie crossed
module. A \textit{derivation} from $L_{0}$ to $L_{1}$ is the $\mathbb{K}$%
-linear function $\partial :L_{0}\longrightarrow L_{1}$ such that 
\[
\partial \lbrack l_{0},l_{0}^{\prime }]=[l_{0},\partial (l_{0}^{\prime
})]-[l_{0}^{\prime },\partial (l_{0})],
\]%
for all $l_{0},l_{0}^{\prime }\in L_{0}.$ The set of all derivations is denoted by $Der(L_{0},L_{1})\ $which carries a
Lie algebra structure with the bracket $[\partial _{1},\partial
_{2}]=\partial _{1}(d\partial _{2})-\partial _{2}(d\partial _{1}),$ for all $%
\partial _{1},$ $\partial _{2}\in Der(L_{0},L_{1})$. On the other hand, \textit{a derivation of a Lie crossed module} $L:L_{1}%
\overset{d}{\longrightarrow }L_{0}$ is a pair $(\alpha ,\beta )$ where $%
\alpha \in Der(L_{1})$ and $\beta \in Der(L_{0})$ such that:%
\[
\begin{array}{ll}
(i) & \beta d=d\alpha \\ 
(ii) & \alpha ([l_{0},l_{1}])=[l_{0},\alpha (l_{1})]+[\beta (l_{0}),l_{1}]%
\end{array}%
\]%
for all $l_{0}\in L_{0}$ and $l_{1}\in L_{1}.$ The set of derivations of the Lie crossed module $L:L_{1}\overset{d}{%
	\longrightarrow }L_{0}$ is denoted by $Der(L).$ $Der(L)$ is endowed with a
Lie $\mathbb{K}$-algebra structure with component-wise addition, scalar
multiplication and the bracket $[(\alpha ,\beta ),(\alpha ^{\prime },\beta
^{\prime })]=([\alpha ,\alpha ^{\prime }],[\beta ,\beta ^{\prime }]),$ for
all $(\alpha ,\beta ),(\alpha ^{\prime },\beta ^{\prime })\in Der(L).$

For a given Lie crossed module $L:L_{1}\overset{d}{\longrightarrow }L_{0},$
we have the corresponding crossed module 
\[
\begin{array}{cccc}
\Delta : & Der(L_{0},L_{1}) & \longrightarrow  & Der(L) \\ 
& \partial  & \longmapsto  & (\partial d,d\partial )%
\end{array}%
\]%
with the action of $Der(L)$ on $Der(L_{0},L_{1})$ given by $\left[ (\alpha
,\beta ),\partial \right] =\alpha \partial -\partial \beta $, for all $%
\left( \alpha ,\beta \right) \in Der(L),$ $\partial \in Der(L_{0},L_{1}).$
This crossed module was introduced in \cite{C3} and is an \textit{%
	actor }of $L$, denoted by $Act(L)$. Also see \cite{tamar1}, for a
comprehensive research about the representability of actions.

\begin{definition}
	\cite{C2} Let $L:L_{1}\overset{d}{\longrightarrow }L_{0}$ be a Lie crossed
	module. Then the center of $L$ is the crossed module $Z(L):L_{1}^{L_{0}}%
	\overset{d\left\vert {}\right. }{\longrightarrow }(St_{L_{0}}(L_{1})\cap
	Z(L_{0}))$ where 
	\[
	L_{1}^{L_{0}}=\{l_{1}\in L_{1}:\text{ }[l_{0},l_{1}]=0,\text{ for all }%
	l_{0}\in L_{0}\}
	\]%
	and 
	\[
	St_{L_{0}}(L_{1})=\{l_{0}\in L_{0}:\text{ }[l_{0},l_{1}]=0,\text{ for all }%
	l_{1}\in L_{1}\}.
	\]%
	On the other hand, the \textit{commutator subcrossed module} $[L,L]$ of $L$
	is defined by 
	\[
	\lbrack L,L]:D_{L_{0}}(L_{1})\overset{d|}{\longrightarrow }[L_{0},L_{0}]
	\]%
	where $D_{L_{0}}(L_{1})=\{[l_{0},l_{1}]:l_{0}\in L_{0},l_{1}\in L_{1}\}$ and 
	$[L_{0},L_{0}]$ is the commutator subalgebra of $L_{0}$.
\end{definition}

The group analogues of following proposition can be found in \cite{norrie1}.

\begin{proposition}
	\label{10} Let $L:L_{1}\overset{d}{\longrightarrow }L_{0}$ be a Lie crossed
	module. Then we have the following;\newline
	(i) If $L$ is simply connected, then $L_{1}^{L_{0}}=Z(L_{1})$ and $%
	D_{L_{0}}(L_{1})=[L_{1},L_{1}]$.\newline
	(ii) If $L$ is aspherical, then $Z(L_{0})=St_{L_{0}}(L_{1})\cap Z(L_{0})$.
\end{proposition}

\begin{proof}
	(i) Let $l_{1}\in L_{1}^{L_{0}}.$ Since $L$ is simply connected, for every $%
	l_{0}\in L_{0}$ there exists $l_{1}^{\prime }\in L_{1}$ such that $%
	d(l_{1}^{\prime })=l_{0}.$ Then $[l_{0},l_{1}]=[d(l_{1}^{\prime }),l_{1}]=0$
	and $[l_{1}^{\prime },l_{1}]=0.$ So $l_{1}\in Z(L_{1})$ i.e $%
	L_{1}^{L_{0}}\subseteq Z(L_{1}).$ Conversely, let $l_{1}\in Z(L_{1}).$ From
	the hypothesis, we have $[l_{0},l_{1}]=[d(l_{1}^{\prime
	}),l_{1}]=[l_{1}^{\prime },l_{1}]=0$. So $l_{1}\in L_{1}^{L_{0}}$ i.e $%
	Z(L_{1})\subseteq L_{1}^{L_{0}}.$ Let $[l_{0},l_{1}]\in D_{L_{0}}(L_{1}).$
	From the hypothesis, we can say that $[l_{0},l_{1}]=[d(l_{1}^{\prime
	}),l_{1}]=[l_{1}^{\prime },l_{1}]\in \lbrack L_{1},L_{1}].$ So we have $%
	D_{L_{0}}(L_{1})\subseteq \lbrack L_{1},L_{1}].$ Let $[l_{1},l_{1}^{\prime
	}]\in \lbrack L_{1},L_{1}].$ Then $[l_{1},l_{1}^{\prime }]=[d(l_{1}^{\prime
}),l_{1}]\in D_{L_{0}}(L_{1})$ i.e $[L_{1},L_{1}]\subseteq D_{L_{0}}(L_{1}).$

(ii) Let $l_{0}\in Z(L_{0}).$ Then we have $%
d([l_{0},l_{1}])=[l_{0},d(l_{1})]=0=d(0).$ Since $L$ is aspherical, $%
[l_{0},l_{1}]=0$ i.e $l_{0}\in St_{L_{0}}(L_{1}).$ So, we have $%
Z(L_{0})\subseteq St_{L_{0}}(L_{1})$ i.e $Z(L_{0})=St_{L_{0}}(L_{1})\cap
Z(L_{0}).$
\end{proof}

\begin{definition}
	A Lie crossed module $L:L_{1}\overset{d}{\longrightarrow }L_{0}$ is called 
	\textit{finite dimensional }if $L_{1}$ and $L_{0}$ are finite dimensional
	Lie algebras.
\end{definition}

\begin{definition}
	\cite{C1} Let $L:L_{1}\overset{d}{\longrightarrow }L_{0}$ be a Lie crossed
	module. If there exists $n\in 
	\mathbb{Z}
	^{+}$ such that $(L_{1},L_{0},d)^{(n)}=0$ then $L$ is called \textit{%
		solvable Lie crossed module. Also, the least positive integer }$n$
	satisfying $(L_{1},L_{0},d)^{(n)}=0$ is called \textit{derived length} of
	the $L$. If there exists $n\in 
	\mathbb{Z}
	^{+}$ such that $(L_{1},L_{0},d)^{n}=0,$ $L$ is called \textit{nilpotent Lie
		crossed module. Also, the least natural }$n$ satisfying $%
	(L_{1},L_{0},d)^{n}=0$ is called \textit{nilpotency class} of $L$.
\end{definition}

\subsection{Class preserving actors of Lie crossed modules}

In this subsection, we introduce the notion of class preserving actor of a
Lie crossed module for constructing relation between isoclinic Lie crossed
modules and their actors. We inspire from Theorem 4.1  \cite{YA} in which
they claim that the class preserving automorphisms groups of isoclinic groups are isomorphic.

\begin{proposition}
	\label{06a} Let $L:L_{1}\overset{d}{\longrightarrow }L_{0}$ be a Lie crossed
	module, $Der_{\mathcal{C}}(L_{0},L_{1}):=\{\delta \in Der(L_{0},l_{1})$ $|$
	there exists $l_{1}\in l_{1}$ such that $\delta (l_{0})=[l_{0},l_{1}],$ for
	all $l_{0}\in L_{0}\}$ and $Der_{_{\mathcal{C}}}(L):=\{(\alpha ,\beta )\in
	Der(L)$ $|$ there exists $l_{0}\in L_{0}$ such that $\alpha (l_{1})=$ $%
	[l_{0},l_{1}],$ $\beta (l_{0}^{\prime })=[l_{0},l_{0}^{\prime }],$ for all $%
	l_{1}\in L_{1},l_{0}^{\prime }\in L_{0}\}$. Then, we have the followings: 
	\newline
	(a) $Der_{\mathcal{C}}(L_{0},L_{1})$ is a Lie subalgebra of $%
	Der(L_{0},L_{1}).$\newline
	(b) $Der_{_{\mathcal{C}}}(L)$ is a Lie subalgebra of $Der(L).$
\end{proposition}

\begin{proof}
	($a$) Let $\delta ,\delta ^{\prime }\in Der_{\mathcal{C}}(L_{0},L_{1}).$ We
	first show that $[\delta ,\delta ^{\prime }]\in Der_{\mathcal{C}%
	}(L_{0},L_{1}).$ Since $\delta ,\delta ^{\prime }\in Der_{\mathcal{C}%
}(L_{0},L_{1}),$ there exist $l_{1},l_{1}^{\prime }\in L_{1}$ such that $%
\delta (l_{0})=[l_{0},l_{1}]$ and $\delta ^{\prime
}(l_{0})=[l_{0},l_{1}^{\prime }],$ for all $l_{0}\in L_{0}.$ Then, 
\[
\begin{array}{lll}
\lbrack \delta ,\delta ^{\prime }](l_{0}) & = & \delta d\delta ^{\prime
}(l_{0})-\delta ^{\prime }d\delta (l_{0}) \\ 
& = & \delta d[l_{0},l_{1}^{\prime }]-\delta ^{\prime }d[l_{0},l_{1}] \\ 
& = & \delta \lbrack l_{0},d(l_{1}^{\prime })]-\delta ^{\prime
}[l_{0},d(l_{1})] \\ 
& = & \left[ [l_{0},d(l_{1}^{\prime })],l_{1}\right] -\left[
[l_{0},d(l_{1})],l_{1}^{\prime }\right] \\ 
& = & \left[ l_{1,}[d(l_{1}^{\prime }),l_{0}]\right] +\left[ l_{1}^{\prime
},[l_{0},d(l_{1})]\right] \\ 
& = & -[d(l_{1}^{\prime }),[l_{0},l_{1}]]-[l_{0},[l_{1},d(l_{1}^{\prime
})]]+[l_{1}^{\prime },[l_{0},d(l_{1})]] \\ 
& = & -[l_{1}^{\prime },[l_{0},l_{1}]]-[l_{0},[l_{1},l_{1}^{\prime
}]]+[l_{1}^{\prime },d[l_{0},l_{1}]] \\ 
& = & -[l_{1}^{\prime },[l_{0},l_{1}]]-[l_{0},[l_{1},l_{1}^{\prime
}]]+[l_{1}^{\prime },[l_{0},l_{1}]] \\ 
& = & [l_{0},[l_{1}^{\prime },l_{1}]]%
\end{array}%
\]%
i.e $[\delta ,\delta ^{\prime }]\in Der_{\mathcal{C}}(L_{0},L_{1}).$ On the
other hand, $\delta -\delta ^{\prime }\in Der_{\mathcal{C}}(L_{0},L_{1})$
and $(k\delta )\in Der_{\mathcal{C}}(L_{0},L_{1}),$ we have for any $\delta
,\delta ^{\prime }\in Der_{\mathcal{C}}(L_{0},L_{1})$, $k\in \mathbb{K}$. 
\newline
($b$) By a similar way, we have $[(\alpha ,\beta ),(\alpha ^{\prime },\beta
^{\prime })],$ $(\alpha ,\beta )-$ $(\alpha ^{\prime },\beta ^{\prime }),$ $%
k(\alpha ,\beta )\in Der_{_{\mathcal{C}}}(L)$ for all $(\alpha ,\beta ),$ $%
(\alpha ^{\prime },\beta ^{\prime })\in Der_{_{\mathcal{C}}}(L),$ $k\in 
\mathbb{K}$, as required.
\end{proof}

\begin{proposition}
	\label{06} Let $L:L_{1}\overset{d}{\longrightarrow }L_{0}$ be a Lie crossed
	module. $Act_{\mathcal{C}}(L):Der_{\mathcal{C}}(L_{0},L_{1})\overset{\Delta
		^{\prime }}{\longrightarrow }Der_{_{\mathcal{C}}}(L)$ is a crossed module
	with the action induced from the action of $Der(L)$ over $Der(L_{0},L_{1})$
	such that 
	\[
	\begin{array}{lll}
	Der_{_{\mathcal{C}}}(L)\times Der_{\mathcal{C}}(L_{0},L_{1}) & 
	\longrightarrow & Der_{\mathcal{C}}(L_{0},L_{1}) \\ 
	((\alpha ,\beta ),\delta ) & \longmapsto & (\alpha ,\beta )\cdot \delta
	=\alpha \delta -\delta \beta .%
	\end{array}%
	\]
\end{proposition}

\begin{proof}
	It can be shown by a direct calculation.
\end{proof}

\begin{definition}
	Let $L:L_{1}\overset{d}{\longrightarrow }L_{0}$ be a Lie crossed module. The
	crossed module 
	\[
	Act_{\mathcal{C}}(L):Der_{\mathcal{C}}(L_{0},L_{1})\overset{\Delta _{_{%
				\mathcal{C}}}}{\longrightarrow }Der_{_{\mathcal{C}}}(L)
	\]%
	defined in Proposition \ref{06} will be called as the \textit{class preserving
		actor} of $L$ and will be denoted by $Act_{\mathcal{C}}(L).$
\end{definition}

\begin{remark}
	For any Lie crossed module $L:L_{1}\overset{d}{\longrightarrow }L_{0},$ we
	have $InnAct(L)\trianglelefteq Act_{\mathcal{C}}(L)\leq Act(L).$
\end{remark}

\section{Isoclinic Lie crossed modules}

In this section, we will introduce the notion of isoclinism among Lie
crossed modules and establish the basic theory.

{\noindent \textbf{Notation}} In the sequel of the paper, for a given Lie
crossed module $L:L_{1}\overset{d}{\longrightarrow }L_{0},$ we denote $%
L/Z(L) $ by $\overline{L_{1}}\overset{\overline{d}}{\longrightarrow }%
\overline{L_{0}}$ where $\overline{L_{1}}=L_{1}/L_{1}^{L_{0}}$ and $%
\overline{L_{0}}=L_{0}/(St_{L_{0}}(L_{1})\cap Z(L_{0}))$, for shortness.

\begin{proposition}
	Let $L:L_{1}\overset{d}{\longrightarrow }L_{0}$ be a Lie crossed module.
	Define the maps 
	\[
	\begin{array}{ccccc}
	c_{1} & : & \overline{L_{1}}\times \overline{L_{0}} & \longrightarrow & 
	D_{L_{0}}(L_{1}) \\ 
	&  & (l_{1}+L_{1}^{L_{0}},l_{0}+(St_{L_{0}}(L_{1})\cap Z(L_{0}))) & 
	\longmapsto & [l_{0},l_{1}]%
	\end{array}%
	\]
	and 
	\[
	\begin{array}{ccccc}
	c_{0} & : & \overline{L_{0}}\times \overline{L_{0}} & \longrightarrow & 
	\left[ L_{0},L_{0}\right] \\ 
	&  & (l_{0}+(St_{L_{0}}(L_{1})\cap Z(L_{0})),l_{0}^{^{\prime
		}}+(St_{L_{0}}(L_{1})\cap Z(L_{0}))) & \longmapsto & [l_{0},l_{0}^{\prime }],%
		\end{array}%
		\]
		for all $l_{1}\in L_{1},$ $l_{0},l_{0}^{\prime }\in L_{0}.$ Then the maps $%
		c_{1}$ and $c_{0}$ are well-defined.
	\end{proposition}
	
	\begin{proof}
		Let $(l_{1}+L_{1}^{L_{0}},l_{0}+St_{L_{0}}(L_{1})\cap
		Z(L_{0}))=(m_{1}+L_{1}^{L_{0}},m_{0}+St_{L_{0}}(L_{1})\cap Z(L_{0})),$ then $%
		-l_{1}+m_{1}\in L_{1}^{L_{0}}$ and $-l_{0}+m_{0}\in St_{L_{0}}(L_{1})\cap
		Z(L_{0})\subseteq Z(L_{0}).$ So, we have 
		\[
		\begin{array}{llll}
		c_{1}(\overline{l_{1}},\overline{l_{0}}) & = & [l_{0},l_{1}] &  \\ 
		& = & [l_{0},l_{1}]+[l_{0},(-l_{1}+m_{1})] & (-l_{1}+m_{1}\in L_{1}^{L_{0}})
		\\ 
		& = & [l_{0},m_{1}] &  \\ 
		& = & [(-l_{0}+m_{0}),m_{1}]+[l_{0},m_{1}]\text{ } & (-l_{0}+m_{0}\in
		St_{L_{0}}(L_{1})) \\ 
		& = & [m_{0},m_{1}] &  \\ 
		& = & c_{1}(\overline{m_{1}},\overline{m_{0}}) & 
		\end{array}%
		\]%
		which gives the well-definition of $c_{1}.$\newline
		Similarly, let $\overline{l_{0}},\overline{m_{0}},\overline{l_{0}^{\prime }},%
		\overline{m_{0}^{\prime }}\in \overline{L_{0}}$. If $(\overline{l_{0}},%
		\overline{l_{0}^{\prime }})=(\overline{m_{0}},\overline{m_{0}^{\prime }}),$
		then $-l_{0}+m_{0},$ $-l_{0}^{\prime }+m_{0}^{\prime }\in
		St_{L_{0}}(L_{1})\cap Z(L_{0})\subseteq Z(L_{0}).$ Then, we have 
		\[
		-[l_{0},l_{0}^{\prime }]+[l_{0},m_{0}^{\prime }]=0 
		\]%
		and 
		\[
		\lbrack l_{0},m_{0}^{\prime }]-[m_{0},m_{0}^{\prime }]=0 
		\]%
		from which, we obtain $c_{0}(\overline{m_{0}},\overline{m_{0}^{\prime }}%
		)=c_{0}(\overline{l_{0}},\overline{l_{0}^{\prime }})$, as required.
	\end{proof}
	
	{\noindent }The pair $(c_{1},c_{0})$ will be called as the commutator map of
	the Lie crossed module $L:L_{1}\overset{d}{\longrightarrow }L_{0}.$
	
	\begin{definition}
		\label{04} The Lie crossed modules $L:L_{1}\overset{d_{L}}{\longrightarrow }%
		L_{0}$ and $L^{\prime }:L_{1}^{\prime }\overset{d_{L^{\prime }}}{%
			\longrightarrow }L_{0}^{\prime }$ are \textit{called isoclinic} if there
		exist isomorphisms 
		\[
		(\eta _{1},\eta _{0}):(\overline{L_{1}}\overset{\overline{d_{L}}}{%
			\longrightarrow }\overline{L_{0}})\longrightarrow (\overline{L_{1}^{\prime }}%
		\overset{\overline{d_{L^{\prime }}}}{\longrightarrow }\overline{%
			L_{0}^{\prime }})
		\]%
		and 
		\[
		(\xi _{1},\xi _{0}):(D_{L_{0}}(L_{1})\overset{d_{L}|}{\longrightarrow }%
		[L_{0},L_{0}])\longrightarrow (D_{L_{0}^{\prime }}(L_{1}^{\prime })\overset{%
			d_{L^{\prime }}|}{\longrightarrow }[L_{0}^{\prime },L_{0}^{\prime }])
		\]%
		such that the diagrams 
		\begin{equation}
\xymatrix { {\overline{L_{1}}} \times {\overline{L_{0}}} \ar[rr]^ -{c_1} \ar[d]_{{{\eta}_1} \times {{\eta}_0}}  & & D_{L_{0}}{(L_{1})} \ar[d]^{{\xi}_1} \\ {\overline{{L_{1}^{\prime }}}} \times {\overline{{L_{0}^{\prime }}}}   \ar[rr]_ -{{c_1}'}  & & D_{L_{0}^{\prime }}{(L_{1}^{\prime }) }}
		\end{equation}%
		and 
		\begin{equation}
\xymatrix { {\overline{L_{0}}} \times {\overline{L_{0}}} \ar[rr]^ -{c_0} \ar[d]_{{{\eta}_0} \times {{\eta}_0}}  & & [L_0,L_0] \ar[d]^{{\xi}_0} \\ {\overline{{L_{0}^{\prime }}}} \times {\overline{{L_{0}^{\prime }}}}   \ar[rr]_ -{{c_0}'}  & & [{L_{0}^{\prime }},{L_{0}^{\prime }}] }
		\end{equation}%
		are commutative where $(c_{1},c_{0})$ and $(c_{1}^{\prime },c_{0}^{\prime })$
		are commutator maps of $L$ and $L^{\prime },$ respectively. The pair $((\eta _{1},\eta _{0}),(\xi _{1},\xi _{0}))$ will be called an \textit{isoclinism} from $L$ to $L^{\prime }$ and this situation will be denoted by $((\eta _{1},\eta _{0}),(\xi _{1},\xi _{0})):L\sim
		L^{\prime }.$
		
	\end{definition}
	
	\begin{examples}
		\ \newline
		(1) All abelian Lie crossed modules (crossed modules coincide with their centers) are isoclinic. The pairs $((\eta
		_{1},\eta _{0}),(\xi _{1},\xi _{0}))$ consist of trivial homomorphisms. 
		\newline
		(2) Let $(\eta ,\xi )$ be an isoclinism between the Lie algebras $\mathfrak{g%
		}$ to $\mathfrak{h}$. Then $\mathfrak{g}\overset{id}{\longrightarrow }\mathfrak{g}$
		is isoclinic to $\mathfrak{h}\overset{id}{\longrightarrow }\mathfrak{h}$,
		where $(\eta _{1},\eta _{0})=(\eta ,\eta )$, $(\xi _{1},\xi _{0})=(\xi ,\xi )
		$.\newline
		(3) Let $\mathfrak{g}$ be a Lie algebra and let $I$ be an ideal
		of $\mathfrak{g}$ with $I+Z(\mathfrak{g})=\mathfrak{g}.$ Then $%
		I \overset{inc.}{\hookrightarrow }\mathfrak{g}$ is isoclinic to $%
		\mathfrak{g}\overset{id}{\longrightarrow }\mathfrak{g}.$ Here $(\eta
		_{1},\eta _{0})$ and $(\xi _{1},\xi _{0})$ are defined by $(inc.,inc.)$, $%
		(id,id)$, respectively.
	\end{examples}
	
	\begin{remark}
		If the Lie crossed modules $L$ and $L^{\prime }$ are simply connected or
		finite dimensional, then the commutativity of diagrams $(1)$ with $(2)$ are
		equivalent to the commutativity of the following diagram. 
		\[
\xymatrix { {L/{Z(L)}} \times {L/{Z(L)}} \ar[rr]^-{} \ar[d]_{({\eta_1}, {\eta_0})\times({\eta_1}, {\eta_0})}  & & [L,L] \ar[d]^{({\xi_1},{\xi_0})} \\ {L^{\prime }/{Z(L^{\prime })}} \times {L^{\prime }/{Z(L^{\prime })}} \ar[rr]_-{}  & & [L^{\prime },L^{\prime }] }
		\]
	\end{remark}
	
	\begin{proposition}
		Isoclinism is an equivalence relation.
	\end{proposition}
	
	\begin{proof}
		It can be easily checked by a direct calculation.
	\end{proof}
	
	\begin{proposition}
		\label{05} Let $L:L_{1}\overset{d}{\longrightarrow }L_{0}$ be a Lie crossed
		module and $M:M_{1}\overset{d|}{\longrightarrow }M_{0}$ be its subcrossed
		module. If $L=M\oplus Z(L),$ i.e, $L_{1}=M_{1}\oplus L_{1}^{L_{0}}$ and $%
		L_{0}=M_{0}\oplus (St_{L_{0}}(L_{1})\cap Z(L_{0})),$ then $L$ is isoclinic
		to $M.$
	\end{proposition}
	
	\begin{proof}
		First, we show that $M_{1}^{M_{0}}=M_{1}\cap L_{1}^{L_{0}}$ and $%
		St_{M_{0}}(M_{1})\cap Z(M_{0})=M_{0}\cap (St_{L_{0}}(L_{1})\cap Z(L_{0})).$ 
		\newline
		Let $m_{1}\in M_{1}^{M_{0}}.$ For any $l_{0}\in L_{0},$ since $%
		L_{0}=M_{0}\oplus (St_{L_{0}}(L_{1})\cap Z(L_{0}))$ there exist $a_{0}\in
		St_{L_{0}}(L_{1})\cap Z(L_{0})$ and $m_{0}^{\prime }\in M_{0}$ such that $%
		l_{0}=m_{0}^{\prime }+a_{0}.$ We have $[l_{0},m_{1}]=[(m_{0}^{\prime
		}+a_{0}),m_{1}]=0$, so $m_{1}\in M_{1}\cap L_{1}^{L_{0}}.$ Conversely, for
		any $m_{1}\in M_{1}\cap L_{1}^{L_{0}},$ we have $m_{1}\in M_{1}^{M_{0}}.$
		So, $M_{1}^{M_{0}}=M_{1}\cap L_{1}^{L_{0}}.$\newline
		Let $m_{0}\in St_{M_{0}}(M_{1})\cap Z(M_{0}).$ For any $l_{1}\in L_{1},$
		there exist $k_{1}\in M_{1}$ and $a_{1}\in L_{1}^{L_{0}}$ such that $%
		l_{1}=k_{1}+a_{1}.$ Then $[m_{0},l_{1}]=[m_{0},(k_{1}+a_{1})]=0,$ which
		means that $m_{0}\in St_{L_{0}}(L_{1}).$ On the other hand, it is clear that 
		$m_{0}\in Z(L_{0}).$ Then, we obtain $m_{0}\in M_{0}\cap
		(St_{L_{0}}(L_{1})\cap Z(L_{0})).$ By a direct calculation, we get $%
		St_{M_{0}}(M_{1})\cap Z(M_{0})=M_{0}\cap (St_{L_{0}}(L_{1})\cap Z(L_{0})).$
		By the second isomorphism theorem for Lie crossed modules, we have 
		\[
		\begin{array}{lll}
		\dfrac{M}{Z(M)} & = & \dfrac{(M_{1},M_{0},d|)}{%
			(M_{1}^{M_{0}},St_{M_{0}}(M_{1})\cap Z(M_{0}),d|)} \\ 
		& = & \dfrac{%
			\begin{array}{c}
			\\ 
			(M_{1},M_{0},d|)%
			\end{array}%
		}{(M_{1}\cap L_{1}^{L_{0}},M_{0}\cap (St_{L_{0}}(L_{1})\cap Z(L_{0})),d|)}
		\\ 
		& = & \dfrac{%
			\begin{array}{c}
			\\ 
			(M_{1},M_{0},d|)%
			\end{array}%
		}{(M_{1},M_{0},d|)\cap (L_{1}^{L_{0}},St_{L_{0}}(L_{1})\cap Z(L_{0})),d|)}
		\\ 
		& \cong  & \dfrac{%
			\begin{array}{c}
			\\ 
			(M_{1},M_{0},d|)\oplus (L_{1}^{L_{0}},St_{L_{0}}(L_{1})\cap Z(L_{0})),d|)%
			\end{array}%
		}{(L_{1}^{L_{0}},St_{L_{0}}(L_{1})\cap Z(L_{0})),d|)} \\ 
		& = & \dfrac{%
			\begin{array}{c}
			\\ 
			M\oplus Z(L)%
			\end{array}%
		}{Z(L)} \\ 
		& = & \dfrac{%
			\begin{array}{c}
			\\ 
			L%
			\end{array}%
		}{Z(L)}.%
		\end{array}%
		\]%
		Let $[l_{0},l_{1}]\in D_{L_{0}}(L_{1}),$ then there exist $m_{1}\in M_{1},$ $%
		a_{1}\in L_{1}^{L_{0}},$ $m_{0}\in M_{0},$ $a_{0}\in (St_{L_{0}}(L_{1})\cap
		Z(L_{0}))$ such that $l_{1}=m_{1}+a_{1}$ and $l_{0}=m_{0}+a_{0}.$ Since 
		\begin{eqnarray*}
			\lbrack l_{0},l_{1}] &=&[(m_{0}+a_{0}),(m_{1}+a_{1})] \\
			&=&[m_{0},(m_{1}+a_{1})]+[a_{0},(m_{1}+a_{1})] \\
			&=&[m_{0},m_{1}]+[m_{0},a_{1}]+[a_{0},m_{1}]+[a_{0},a_{1}] \\
			&=&[m_{0},m_{1}]\text{,}
		\end{eqnarray*}%
		we have $[l_{0},l_{1}]\in D_{M_{0}}(M_{1}).$ On the other hand, for any $%
		[l_{0},l_{0}^{\prime }]\in \lbrack L_{0},L_{0}]$ there exist $%
		m_{0},m_{0}^{\prime }\in M_{0},$ $a_{0},a_{0}^{\prime }\in
		(St_{L_{0}}(L_{1})\cap Z(L_{0}))$ such that $l_{0}=m_{0}+a_{0},$ $%
		l_{0}^{\prime }=m_{0}^{\prime }+a_{0}^{\prime },$ from which we get 
		\begin{eqnarray*}
			\lbrack l_{0},l_{0}^{\prime }] &=&[m_{0}+a_{0},m_{0}^{\prime }+a_{0}^{\prime
			}] \\
			&=&[m_{0},m_{0}^{\prime }]+[a_{0},m_{0}^{\prime }]+[m_{0},a_{0}^{\prime
			}]+[a_{0},a_{0}^{\prime }] \\
			&=&[m_{0},m_{0}^{\prime }].
		\end{eqnarray*}%
		Finally, the Lie crossed modules $L$ and $M$ are isoclinic where the
		isomorphisms $(\eta _{1},\eta _{0})$ and $(\xi _{1},\xi _{0})$ are defined
		by $(inc.,inc.)$, $(id,id)$, respectively.
	\end{proof}
	
	\begin{remark}
		If $M:M_{1}\overset{d|}{\longrightarrow }M_{0}$ is a finite dimensional,
		then the converse of Proposition \ref{05} is true.
	\end{remark}
	
	\begin{proposition}
		Let $L:L_{1}\overset{d_{L}}{\longrightarrow }L_{0}$ and $L^{\prime
		}:L_{1}^{\prime }\overset{d_{L^{\prime }}}{\longrightarrow }L_{0}^{\prime }$
		be isoclinic crossed mo\-dules.\newline
		(i) If $L$ and $L^{\prime }$ are aspherical, then $L_{0}$ and $L_{0}^{\prime
		}$ are isoclinic Lie algebras$.$\newline
		(ii) If $L$ and $L^{\prime }$ are simply connected, then $L_{1}$ and $%
		L_{1}^{\prime }$ are isoclinic Lie algebras$.$
	\end{proposition}
	
	\begin{proof}
		Let $L:L_{1}\overset{d_{L}}{\longrightarrow }L_{0}$ and $L^{\prime
		}:L_{1}^{\prime }\overset{d_{L^{\prime }}}{\longrightarrow }L_{0}^{\prime }$
		be isoclinic Lie crossed modules. Then we have the isomorphisms 
		\begin{eqnarray*}
			(\eta _{1},\eta _{0}) &:&(\overline{L_{1}}\overset{\overline{d_{L}}}{%
				\longrightarrow }\overline{L_{0}})\longrightarrow (\overline{L_{1}^{\prime }}%
			\overset{\overline{d_{L^{\prime }}}}{\longrightarrow }\overline{%
				L_{0}^{\prime }}) \\
			(\xi _{1},\xi _{0}) &:&(D_{L_{0}}(L_{1})\overset{d_{L}|}{\longrightarrow }%
			[L_{0},L_{0}])\longrightarrow (D_{L_{0}^{\prime }}(L_{1}^{\prime })\overset{%
				d_{L^{\prime }}|}{\longrightarrow }[L_{0}^{\prime },L_{0}^{\prime }])
		\end{eqnarray*}%
		which makes diagrams (1) and (2) commutative.
		
		(i) Since $L$ and $L^{\prime }$ are aspherical, we have $Z(L_{0})\subseteq
		St_{L_{0}}(L_{1}),$ $Z(L_{0}^{\prime })\subseteq St_{L_{0}^{\prime
			}}(L_{1}^{\prime }).$ Consequently, $\eta _{0}$ is an isomorphism between $%
			L_{0}/Z(L_{0})$ and $L_{0}^{\prime }/Z(L_{0}^{\prime }).$ So the pair $%
			\left( \eta _{0},\xi _{0}\right) $ is an isoclinism from $L_{0}$ to $%
			L_{0}^{\prime }.$
			
			(ii) Since $L$ and $L^{\prime }$ are simply connected, we have $%
			L_{1}^{L_{0}}=Z(L_{1}),$ $L_{1}^{\prime L_{0}^{\prime }}=Z(L_{1}^{\prime }),$
			$D_{L_{0}}(L_{1})=[L_{1},L_{1}]$ and $D_{L_{0}^{\prime }}(L_{1}^{\prime
			})=[L_{1}^{\prime },L_{1}^{\prime }].$ So we have the isomorphisms $\eta
			_{1}:L_{1}/Z(L_{1})\longrightarrow L_{1}^{\prime }/Z(L_{1}^{\prime }),$ $\xi
			_{1}:[L_{1},L_{1}]\longrightarrow \lbrack L_{1}^{\prime },L_{1}^{\prime }].$
			The pair $\left( \eta _{1},\xi _{1}\right) $ is an isoclinism from $L_{1}$
			to $L_{1}^{\prime },$ as required.
		\end{proof}
		
		\begin{proposition} \label{prop23}
			Let $L:L_{1}\overset{d_{L}}{\longrightarrow }L_{0}$ and $L^{\prime
			}:L_{1}^{\prime }\overset{d_{L^{\prime }}}{\longrightarrow }L_{0}^{\prime }$
			be isoclinic finite dimensional Lie crossed modules. Then $L_{1}$ and $L_{0}$
			are isoclinic to $L_{1}^{\prime }$ and $L_{0}^{\prime },$ respectively.
		\end{proposition}
		
		\begin{proof}
			Let $L:L_{1}\overset{d_{L}}{\longrightarrow }L_{0}$ and $L^{\prime
			}:L_{1}^{\prime }\overset{d_{L^{\prime }}}{\longrightarrow }L_{0}^{\prime }$
			be isoclinic Lie crossed modules. Then we have the crossed module
			isomorphisms 
			\begin{eqnarray*}
				(\eta _{1},\eta _{0}) &:&(\overline{L_{1}}\overset{\overline{d_{L}}}{%
					\longrightarrow }\overline{L_{0}})\longrightarrow (\overline{L_{1}^{\prime }}%
				\overset{\overline{d_{L^{\prime }}}}{\longrightarrow }\overline{%
					L_{0}^{\prime }}) \\
				(\xi _{1},\xi _{0}) &:&(D_{L_{0}}(L_{1})\overset{d_{L}|}{\longrightarrow }%
				[L_{0},L_{0}])\longrightarrow (D_{L_{0}^{\prime }}(L_{1}^{\prime })\overset{%
					d_{L^{\prime }}|}{\longrightarrow }[L_{0}^{\prime },L_{0}^{\prime }])
			\end{eqnarray*}%
			which makes diagrams (1) and (2) commutative. Since $L_{1}$ and $%
			L_{1}^{\prime }$ are finite dimensional, the restriction $\xi
			_{1}|:[L_{1},L_{1}]\longrightarrow \lbrack L_{1}^{\prime },L_{1}^{\prime }]$
			is also an isomorphism. Similarly, we have the isomorphisms $\eta
			_{1}^{\prime }:L_{1}/Z(L_{1})\longrightarrow L_{1}^{\prime }/Z(L_{1}^{\prime
			}),\eta _{1}^{\prime }(l_{1}Z(L_{1}))=l_{1}^{\prime }Z(L_{1}^{\prime })$, $%
			\eta _{0}^{\prime }:L_{0}/Z(L_{0})\longrightarrow L_{0}^{\prime
			}/Z(L_{0}^{\prime }),\eta _{0}^{\prime }(l_{0}Z(L_{0}))=l_{0}^{\prime
		}Z(L_{0}^{\prime })$, and $\xi _{0}$ which make $L_{1}$ and $L_{0}$
		isoclinic to $L_{1}^{\prime }$ and $L_{0}^{\prime },$ respectively.
	\end{proof}
	
	We will give the relation between the class preserving actors of Lie crossed
	modules and isoclinism. 
	
	\begin{theorem}
		\label{teo} Let $L:L_{1}\overset{d}{\longrightarrow }L_{0}$ and $K:K_{1}%
		\overset{d_{K}}{\longrightarrow }K_{0}$ be isoclinic Lie crossed modules.
		Then, we have $Der_{\mathcal{C}}(L_{0},L_{1})\cong Der_{\mathcal{C}%
		}(K_{0},K_{1}).$
	\end{theorem}
	
	\begin{proof}
		We have the isomorphism,%
		\[
		\begin{array}{llll}
		\phi : & Der_{\mathcal{C}}(L_{0},L_{1}) & \longrightarrow & Der_{\mathcal{C}%
		}(K_{0},K_{1}) \\ 
		& \multicolumn{1}{c}{\delta} & \multicolumn{1}{c}{\longmapsto} & 
		\multicolumn{1}{c}{\lambda}%
		\end{array}%
		\]%
		defined as follows;
		
		Since $L$ and $K$ are isoclinic, there exist isomorphisms 
		\begin{eqnarray*}
			(\eta _{1},\eta _{0}) &:&(\overline{L_{1}}\overset{d_{L}|}{\longrightarrow }%
			\overline{L_{0}})\longrightarrow (\overline{K_{1}}\overset{d_{K}|}{%
				\longrightarrow }\overline{K_{0}}) \\
			(\xi _{1},\xi _{0}) &:&(D_{L_{0}}(L_{1})\overset{d_{L}|}{\longrightarrow }%
			[L_{0},L_{0}])\longrightarrow (D_{K_{0}}(K_{1})\overset{d_{K}|}{%
				\longrightarrow }[K_{0},K_{0}])
		\end{eqnarray*}%
		which make the diagrams (1) and (2) commutative. Let $\delta \in Der_{%
			\mathcal{C}}(L_{0},L_{1})$, $k_{0}\in K_{0}-(St_{K_{0}}(K_{1})\cap Z(K_{0}))$
		and $\overline{k_{0}}\in \overline{K_{0}}$. Set $\overline{l_{0}}=\eta
		_{0}^{-1}(\overline{k_{0}})$. On the other hand, there exists $a_{1}\in L_{1}
		$ such that $\delta (l_{0})=\left[ l_{0},a_{1}\right] $. Let $\eta _{1}(%
		\overline{a_{1}})=\overline{a_{1}^{\prime }}.$ Define 
		\[
		\lambda (k_{0})=\left\{ 
		\begin{array}{lllc}
		\lbrack k_{0},a_{1}^{\prime }], &  & k_{0}\in K_{0}-(St_{K_{0}}(K_{1})\cap
		Z(K_{0})) &  \\ 
		0, &  & k_{0}\in St_{K_{0}}(K_{1})\cap Z(K_{0}). & 
		\end{array}%
		\right. 
		\]
	\end{proof}
	
	\begin{proposition}
		\label{pro}If $L$ and $K$ are finite dimensional non-abelian isoclinic Lie
		crossed modules, then $Der_{_{\mathcal{C}}}(L)\cong Der_{_{\mathcal{C}}}(K).$
	\end{proposition}
	
	\begin{proof}
		Since $L$ and $K$ are isoclinic and finite dimensional, from Proposition \ref{prop23} \ $L_{1}$ and $L_{0}$ are isoclinic to $K_{1}$ and $K_{0}$, respectively. So $Der_{_{\mathcal{C}}}(L_1)\cong Der_{_{\mathcal{C}}}(K_1)$ and $Der_{_{\mathcal{C}}}(L_0)\cong Der_{_{\mathcal{C}}}(K_0)$ from \cite{SM1}, as required.
	\end{proof}
	
	Consequently we have the following result from Theorem \ref{teo} and Proposition \ref{pro}.
	
	\begin{corollary}
		Let $L:L_{1}\overset{d_{L}}{\longrightarrow }L_{0}$ and $K:K_{1}\overset{%
			d_{K}}{\longrightarrow }K_{0}$ be finite dimensional non-abelian isoclinic
		Lie crossed modules. Then $Act_{\mathcal{C}}(L)\cong Act_{\mathcal{\ C}}(K).$
	\end{corollary}

	As expected, we have the following proposition which shows the relation
	between isoclinism, nilpotency and solvability. We omit the proof which is
	an analogue to corollary 14 in \cite{AO}.
	
	\begin{proposition}
		Let $L:L_{1}\overset{d_{L}}{\longrightarrow }L_{0}$ and $K:K_{1}\overset{%
			d_{K}}{\longrightarrow }K_{0}$ be isoclinic Lie crossed modules.\newline
		(i) $L$ is nilpotent (solvable) if and only if $K$ is nilpotent (solvable).%
		\newline
		(ii) If $L$ and $K$ are nilpotent (solvable) and both nontrivial, then they
		have the same nilpotency class (derived length).
	\end{proposition}
	
	\begin{remark}
		When we consider the Lie algebras as crossed modules, then we recover
		classical results for isoclinic Lie algebras. In fact, if $L\overset{id}{%
			\longrightarrow }L$ and $K\overset{id}{\longrightarrow }K$ are isoclinic Lie
		crossed modules then we find that $L$ and $K$ are isoclinic. On the other
		hand, let $\mathfrak{g,}$ $\mathfrak{h}$ are finite dimensional Lie algebras
		and $\mathfrak{g}^{\prime }\trianglelefteq \mathfrak{g}$, $\mathfrak{h}%
		^{\prime }\trianglelefteq \mathfrak{h}$. Then the isoclinism of inclusion
		crossed modules $\mathfrak{g}^{\prime }\overset{inc.}{\hookrightarrow }%
		\mathfrak{g}$ and $\mathfrak{h}^{\prime }\overset{inc.}{\hookrightarrow }%
		\mathfrak{h}$ give rise to the isoclinism between the pair Lie algebras $(%
		\mathfrak{g}^{\prime },\mathfrak{g})$ and $(\mathfrak{h}^{\prime },\mathfrak{%
			h}).$ Also, the converse is true, that is, if $(\mathfrak{g}_{1},\mathfrak{g}%
		_{0})$ and $(\mathfrak{h}_{1},\mathfrak{h}_{0})$ are isoclinic pair Lie
		algebras, then the resulting inclusion crossed modules $\mathfrak{g}_{1}%
		\overset{inc.}{\hookrightarrow }\mathfrak{g}_{0}${\ }and $\mathfrak{h}_{1}%
		\overset{inc.}{\hookrightarrow }\mathfrak{h}_{0}${\ }are isoclinic. (See 
		\cite{MP}, for the definition of isoclinic pair Lie algebras)
	\end{remark}

\end{document}